\documentclass[12pt]{article}

\usepackage{amsmath, epsfig, cite}
\usepackage{amssymb}
\usepackage{amsfonts}
\usepackage{latexsym}
\usepackage{amsthm}
\usepackage{enumerate}

\makeatletter
\renewcommand{\@seccntformat}[1]{{\csname the#1\endcsname}.\hspace{.5em}}
\makeatother

\newtheorem{thm}{Theorem}[section]

\newtheorem{cor}[thm]{Corollary}

\newtheorem{lem}[thm]{Lemma}

\newcommand{\pf}{\noindent{\it Proof.} }

\renewcommand{\qed}{\hfill$\Box$\medskip}

\numberwithin{equation}{section}

\begin{document}

\renewcommand{\thefootnote}{*}

\begin{center}
{\Large\bf Factors of certain sums involving central $q$-binomial\\[5pt] coefficients}
\end{center}

\vskip 2mm \centerline{Victor J. W. Guo$^1$ and Su-Dan Wang$^{2}$\footnote{Corresponding author.}}
\begin{center}
{\footnotesize $^1$School of Mathematics and Statistics, Huaiyin Normal University, Huai'an 223300, Jiangsu,\\
 People's Republic of China\\
{\tt jwguo@hytc.edu.cn}\\[10pt]

$^2$College of Mathematics Science, Inner Mongolia Normal University, Huhhot 010022, Inner Mongolia,\\
 People's Republic of China\\
 {\tt sdwang@imnu.edu.cn}
 }
\end{center}

\vskip 0.7cm \noindent{\small{\bf Abstract.}}  Recently, Ni and Pan proved a $q$-congruence on certain sums involving central $q$-binomial coefficients,
which was conjectured by Guo. In this paper, we give a generalization of  this $q$-congruence and confirm another $q$-congruence,
also conjectured by Guo. Our proof uses Ni and Pan's technique and a simple $q$-congruence observed by Guo and Schlosser.



\section{Introduction}
In 1914, Ramanujan \cite{Ramanujan} obtained a number of representations for $1/\pi$.  One such instance,
though not listed in \cite{Ramanujan}, is the identity
\begin{align}
\sum^{\infty}_{k=0}\frac{4k+1}{(-64)^k}{2k\choose k}^3 =\frac{2}{\pi}, \label{eq:Ramanujan}
\end{align}
which was first proved by Bauer \cite{Bauer} in 1859. Such formulas for $1/\pi$ gained popularity in 1980's
for the reason that they can be utilized to provide fast algorithms for
computing decimal digits of $\pi$. See, for example, the Borwein brothers' monograph \cite{BB}.
Recently, Guillera \cite{Guillera} gave a general method to prove Ramanujan-type series.
For a $q$-analogue of \eqref{eq:Ramanujan},  see \cite{GuoZu0}.

In 1997, Van Hamme \cite{VHamme} conjectured that 13 Ramanujan-type series possess nice $p$-adic analogues, such as
\begin{align}
\sum^{(p-1)/2}_{k=0}\frac{4k+1}{(-64)^k}{2k\choose k}^3 \equiv p(-1)^{(p-1)/2} \pmod{p^3},\label{eq:VAN}
\end{align}
where $p$ is an odd prime. The congruence (1.2) was confirmed by Mortenson \cite{Mortenson} using a $_6F_5$ transformation, and
was reproved by Zudilin \cite{Zud2009} employing the WZ (Wilf--Zeilberger) method. In 2013, also via the WZ method, Sun \cite{Sun2013}
gave the following generalization of \eqref{eq:VAN}: for any integer $n\geqslant 2$,
\begin{align}
\sum^{n-1}_{k=0}(4k+1){2k\choose k}^3(-64)^{n-k-1}\equiv0
\pmod{2n{2n\choose n}}.\label{eq:Sun1-old}
\end{align}
Recently, among other things,
Ni and Pan \cite{NP} proved the following extension of \eqref{eq:Sun1-old}: for each $n\geqslant 2$ and $r\geqslant 1$,
\begin{align}
\sum^{n-1}_{k=0}(4k+1){2k\choose k}^r(-4)^{r(n-k-1)}\equiv0
\pmod{2^{r-2}n{2n\choose n}}.\label{eq:Sun1}
\end{align}

In fact, Ni and Pan also gave a $q$-analogue of \eqref{eq:Sun1}:
for each $n\geqslant2$ and $r\geqslant2$, modulo $(1+q^{n-1})^{2r-2}[n]{2n-1\brack n-1}$,
\begin{align}
\sum_{k=0}^{n-1}(-1)^kq^{k^2+(r-2)k}[4k+1]{2k\brack k}^{2r-1}
(-q^{k+1};q)_{n-k-1}^{4r-2}\equiv 0 , \label{eq:guo0}\\
\frac{1}{1+q^{n-1}}\sum _{k=0}^{n-1}q^{(r-2)k}[4k+1]{2k\brack k}^{2r}
(-q^{k+1};q)_{n-k-1}^{4r} \equiv 0. \label{eq:guo00}
\end{align}
Here $(a;q)_n=(1-a)(1-aq)\cdots(1-aq^{n-1})$ stands for the $q$-\emph{shifted factorial}, and
$[n]=[n]_q=1+q+\cdots+q^{n-1}$ is the \emph{$q$-integer}. The $q$-congruences \eqref{eq:guo0} and \eqref{eq:guo00} were originally conjectured by the first author
\cite[Conjecture 5.4]{Guo2018}. They are obviously true for $r=1$ (the left-hand sides have closed forms; see \cite{Guo2018}).
The $r=2$ case of \eqref{eq:guo0} is a $q$-analogue of \eqref{eq:Sun1} and was first proved by the first author himself.
The $r=2$ case of \eqref{eq:guo00} was first obtained by the present authors \cite{GW}.

The first aim of this paper is to prove the following generalizations of \eqref{eq:guo0} and \eqref{eq:guo00}.
\begin{thm}\label{conj:guo}
For each $n\geqslant 2$, $r\geqslant 1$ and $s\geqslant 0$, modulo $(1+q^{2n-2})^{2r-2}[n]_{q^2}{2n-1\brack n-1}_{q^2}$,
\begin{align}
\sum_{k=0}^{n-1}(-1)^kq^{2k^2+(2r-4s-4)k}[4k+1]^{2s}[4k+1]_{q^2}{2k\brack k}^{2r-1}_{q^2}(-q^{2k+2};q^2)_{n-k-1}^{4r-2}
\equiv 0 \label{eq:guo1},\\
\frac{1}{1+q^{2n-2}}\sum_{k=0}^{n-1}q^{(2r-4s-4)k}[4k+1]^{2s}[4k+1]_{q^2}{2k\brack k}^{2r}_{q^2}(-q^{2k+2};q^2)_{n-k-1}^{4r}
\equiv 0. \label{eq:guo2}
\end{align}
\end{thm}

In particular, letting $q\to 1$ in \eqref{eq:guo1} and \eqref{eq:guo2}, we obtain a generalization of \eqref{eq:Sun1}.
\begin{cor}
For each $n\geqslant 2$, $r\geqslant 1$ and $s\geqslant 0$,
\begin{align*}
\sum^{n-1}_{k=0}(4k+1)^{2s+1}{2k\choose k}^r(-4)^{r(n-k-1)}\equiv0
\pmod{2^{r-2}n{2n\choose n}}.
\end{align*}
\end{cor}

The first author\cite[Theorem 1.1]{Guo2020} proved that, for any positive odd integer $n$,
\begin{align}
\sum_{k=0}^{n-1}[3k+1]\frac{(q;q^2)_k^3q^{-{k+1\choose 2}}}{(q;q)_k^2(q^2;q^2)_k}\equiv[n]q^{(1-n)/2} \pmod{[n]\Phi_n(q)^2}, \label{eq:0phi-q}
\end{align}
where $\Phi_n(q)$ is the $n$-th \emph{cyclotomic polynomial} in $q$, i.e.,
$$\Phi_n(q)=\prod_{\substack{1\leqslant k\leqslant n\\\gcd(n,k)=1}}^n(q-\zeta^k)$$
with $\zeta$ being an $n$-th primitive root of unity.

The second aim of this paper is to prove the following $q$-congruence, which was originally conjectured by the first author \cite[Conjecture 1.7]{Guo2020}.

\begin{thm}\label{conj:guo2} Let $n\geqslant 2$ be an integer. Then
\begin{align}
\sum_{k=0}^{n-1}[3k+1]{2k\brack k}^3(-q^{k+1};q)_{n-k-1}^4q^{-{k+1\choose 2}}\equiv0\pmod{(1+q^{n-1})^2[n]{2n-1\brack n-1}}. \label{eq:conj-guo2}
\end{align}
\end{thm}

Letting $q\to 1$ in \eqref{eq:conj-guo2}, we get the following conclusion, which was conjectured by Sun \cite[Conjecture 5.1(i)]{Sun} and was recently
proved by Mao and Zhang \cite{MZ}.
\begin{cor}For each $n\geqslant 2$,
\begin{align*}
\sum_{k=0}^{n-1}(3k+1){2k\choose k}^3 16^{n-k-1}\equiv0 \pmod{2n{2n\choose n}}.
\end{align*}
\end{cor}

We point out that some other interesting congruences and $q$-congruences can be found in
\cite{Guo2020A.2,Guo-rima,Guo-mod4,GS3,GS,GuoZu,HMZ,Liu,LW,LP,NP2,WY,WY0,Zudilin2}.

\section{Proof of Theorem \ref{conj:guo}}
We need the following lemma, which is a special case of \cite[Lemma 3.2]{NP} and can also be
deduced from the $q$-Lucas theorem (see \cite{Olive}).
\begin{lem}\label{lem:NP}
Let $s$ and $t$ be non-negative integers with $0\leqslant t\leqslant d-1$. Then
\begin{align*}
\frac{(q;q^2)_{sd+t}}{(q^2,q^2)_{sd+t}}
\equiv \frac{1}{4^s}{2s\choose s}
\frac{(q;q^2)_{t}}{(q,q^2)_{t}}
\pmod{\Phi_d(q)}.
\end{align*}
\end{lem}

We adopt some notation similar to that used by Ni and Pan \cite{NP}. For any positive integer $n$, let
\begin{align*}
\mathcal{S}(n)=\left\{d\geqslant 3:  \text{$d$ is odd and}\ \left\lfloor\frac{n-\frac{d+1}{2}}{d}\right\rfloor
=\left\lfloor\frac{n}{d}\right\rfloor\right\},
\end{align*}
where $\left\lfloor x\right\rfloor$ denotes the greatest integer not exceeding $x$. It is easy to see that,
for any integer $d>2n-1$, the number $(d+1)/2$ is greater than $n$,
and so $d\notin\mathcal{S}(n)$. This means that $\mathcal{S}(n)$ only contains finite elements.
Let
\begin{align*}
A_n(q)=\prod_{d\in\mathcal{S}(n)}\Phi_d(q),\\[5pt]
C_n(q)=\prod_{{d\,|\,n,\,d>1}\atop\text{$d$ is odd}}\Phi_d(q),
\end{align*}
It is clear that, if $d\,|\, n$, then $d\notin \mathcal{S}(n)$.
Therefore, the polynomials $A_n(q)$ and $C_n(q)$ are relatively prime.

We now give the key lemma, which is similar to \cite[Theorem 2.1]{NP}.
\begin{lem} \label{thm:NPq2}
Let $\nu_0(q),\nu_1(q),\ldots$ be a sequence of rational functions in $q$
such that, for any positive odd integer $d>1$,
\begin{itemize}
\item[{\rm(i)}] $\nu_k(q)$ is $\Phi_d(q^2)$-integral for each $k\geqslant0$, i.e.,
the denominator of $\nu_k(q)$ is relatively prime to $\Phi_d(q^2)$;

\item[{\rm(ii)}] for any non-negative integers $s$ and $t$ with $ 0\leqslant t\leqslant d-1$,
\begin{align*}
\nu_{sd+t}(q)\equiv\mu_s(q)\nu_{t}(q)
\pmod{\Phi_d(q^2)},
\end{align*}
where $\mu_s(q)$ is a $\Phi_d(q^2)$-integral rational function only dependent on $s$;

\item[{\rm(iii)}]
\begin{align*}
\sum_{k=0}^{(d-1)/2}\frac{(q^2;q^4)_k}{(q^4;q^4)_k}\nu_k(q)\equiv 0 \pmod{\Phi_d(q^2)}.
\end{align*}
\end{itemize}
Then, for all positive integers $n$,
\begin{align}
\sum_{k=0}^{n-1}\frac{(q^2;q^4)_k}{(q^4;q^4)_k}\nu_k(q)\equiv 0
\pmod{A_n(q^2)C_n(q^2)}.  \label{eq:main-lem}
\end{align}
\end{lem}
\pf Our proof is very similar to that of \cite[Theorem 2.1]{NP}.
In order to make the paper more readable, we provide it here. For $d\in\mathcal{S}(n)$, we can write $n=ud+v$ with $(d+1)/2\leqslant v\leqslant d-1$.
Thus, for any $v\leqslant t\leqslant d-1$, the expression $(q^2;q^4)_t/(q^4;q^4)_t$
is congruent to $0$ modulo $\Phi_d(q^2)$.
In view of Lemma~\ref{lem:NP}, we obtain
\begin{align*}
\frac{(q^2;q^4)_{ud+t}}{(q^4;q^4)_{ud+t}}\equiv0\pmod{\Phi_d(q^2)}.
\end{align*}
Applying Lemma \ref{lem:NP} again, we conclude that
\begin{align*}
\sum_{k=0}^{n-1}\frac{(q^2;q^4)_k}{(q^4;q^4)_k}\nu_k(q)
&\equiv
\sum_{s=0}^{u}\sum_{t=0}^{d-1}\frac{(q^2;q^4)_{sd+t}}{(q^4;q^4)_{sd+t}}\nu_{sd+t}(q)\\[5pt]
&\equiv\sum_{s=0}^u \frac{1}{4^s}{2s\choose s} \mu_s(q)
\sum_{t=0}^{d-1}\frac{(q^2;q^4)_{t}}{(q^4;q^4)_{t}}\nu_t(q)\equiv0
\pmod{\Phi_d(q^2)},
\end{align*}
where we have used the condition (iii) and the fact that $(q^2;q^4)_t \nu_t(q)/(q^4;q^4)_t$
is congruent to $0$ modulo $\Phi_d(q^2)$ for $(d+1)/2\leqslant t\leqslant d-1$.
This proves that \eqref{eq:main-lem} is true modulo $A_n(q^2)$.

On the other hand, for $d\,|\,n$, we assume that $u=n/d$. Still by Lemma~\ref{lem:NP}, we have
\begin{align*}
\sum_{k=0}^{n-1}\frac{(q^2;q^4)_k}{(q^4;q^4)_k}\nu_k(q)
&=
\sum_{s=0}^{u-1}\sum_{t=0}^{d-1}\frac{(q^2;q^4)_{sd+t}}{(q^4;q^4)_{sd+t}}\nu_{sd+t}(q)\\[5pt]
&\equiv\sum_{s=0}^{u-1}\frac{1}{4^s}{2s\choose s} \mu_s(q)
\sum_{t=0}^{d-1}\frac{(q^2;q^4)_{t}}{(q^4;q^4)_{t}}\nu_t(q)\\[5pt]
&\equiv0\pmod{\Phi_d(q^2)}.
\end{align*}
This proves that \eqref{eq:main-lem} is also true modulo $C_n(q^2)$.
\qed

We also require the following easily proved result, which is due to Guo and Schlosser \cite[Lemma 3.1]{GS}.
\begin{lem}\label{lem:gs}
Let $d$ be a positive odd integer.
Then, for $0\leqslant k\leqslant (d-1)/2$, we have
\begin{equation*}
\frac{(q;q^2)_{(d-1)/2-k}}{(q^2;q^2)_{(d-1)/2-k}}
\equiv (-1)^{(d-1)/2-2k}\frac{(q;q^2)_k}{(q^2;q^2)_k} q^{(d-1)^2/4+k}
\pmod{\Phi_d(q)}.
\end{equation*}
\end{lem}

In order to simplify the proof of Theorem~\ref{conj:guo}, we need to present the following result.
\begin{thm}\label{conj:guo-new}
Let $n\geqslant 2$, $r\geqslant 1$ and $s\geqslant 0$ be integers. Then
\begin{align}
\sum_{k=0}^{n-1}(-1)^k q^{2k^2+(2r-4s-4)k}[4k+1]^{2s}[4k+1]_{q^2}
\frac{(q^2;q^4)_k^{2r-1}}{(q^4;q^4)_k^{2r-1}}
\equiv 0 \pmod{A_n(q^2)C_n(q^2)}, \label{eq:guo1-new}\\[5pt]
\sum_{k=0}^{n-1}q^{(2r-4s-4)k}[4k+1]^{2s}[4k+1]_{q^2}
\frac{(q^2;q^4)_k^{2r}}{(q^4;q^4)_k^{2r}}
\equiv 0 \pmod{A_n(q^2)C_n(q^2)}. \label{eq:guo2-new}
\end{align}
\end{thm}

\begin{proof}We only give a proof of \eqref{eq:guo1-new}, since the proof of \eqref{eq:guo2-new}
is exactly the same as that of \eqref{eq:guo1-new}.
For any non-negative integer $k$, let
\begin{align*}
\nu_{k}(q)=(-1)^k q^{2k^2+(2r-4s-4)k}[4k+1]^{2s}[4k+1]_{q^2}
\frac{(q^2;q^4)_k^{2r-2}}{(q^4;q^4)_k^{2r-2}}.
\end{align*}
Then, for any odd $d$, the rational function $\nu_{k}(q)$ is $\Phi_d(q^2)$-integral,
since
$$
\frac{(q^2;q^4)_k}{(q^4;q^4)_k}={2k\brack k}_{q^2}\frac{1}{(-q^2;q^2)_k^2},
$$
and $(-q^2;q^2)_k$ is relatively prime to $\Phi_d(q^2)$. Using $q^{2d}\equiv 1\pmod{\Phi_d(q^2)}$ and applying Lemma \ref{lem:NP} with $q\mapsto q^2$, we have
\begin{align*}
\nu_{sd+t}(q)&= (-1)^{sd+t}q^{2(sd+t)^2+(2r-4s-4)(sd+t)}[4(sd+t)+1]^{2s}[4(sd+t)+1]_{q^2}
\frac{(q^2;q^4)_{sd+t}^{2r-2}}{(q^4;q^4)_{sd+t}^{2r-2}}\\
&\equiv (-1)^s\frac{1}{4^{(2r-2)s}}{2s\choose s}^{2r-2} \nu_{t}(q)
\pmod{\Phi_d(q^2)},
\end{align*}
for non-negative integers $s$ and $t$ with $0\leqslant t\leqslant d-1$. Thus, the sequence $\nu_0(q),\nu_1(q),\ldots$ satisfies the requirements (i) and (ii) of Lemma~\ref{thm:NPq2}.

We now verify the requirement (iii) of Lemma~\ref{thm:NPq2}, i.e.,
\begin{equation}
\sum_{k=0}^{(d-1)/2}(-1)^k q^{2k^2+(2r-4s-4)k}[4k+1]^{2s}[4k+1]_{q^2}
\frac{(q^2;q^4)_k^{2r-1}}{(q^4;q^4)_k^{2r-1}}
\equiv 0 \pmod{\Phi_d(q^2)}.  \label{eq:sum-to-0}
\end{equation}
In fact, by Lemma~\ref{lem:gs} with $q\mapsto q^2$, it is easy to verify that, for $0\leqslant k\leqslant (d-1)/2$,
the $k$-th and $((d-1)/2-k)$-th terms on the left-hand side of \eqref{eq:sum-to-0}
cancel each other modulo $\Phi_d(q^2)$, i.e.,
\begin{align*}
&(-1)^{(d-1)/2-k} q^{2((d-1)/2-k)^2+(2r-4s-4)((d-1)/2-k)}\\
&\times [2d-4k-1]^{2s}[2d-4k-1]_{q^2}\frac{(q^2;q^4)_{(d-1)/2-k}^{2r-1}}{(q^4;q^4)_{(d-1)/2-k}^{2r-1}} \notag\\[5pt]
&\quad\equiv -(-1)^k q^{2k^2+(2r-4s-4)k}[4k+1]^{2s}[4k+1]_{q^2}
\frac{(q^2;q^4)_k^{2r-1}}{(q^4;q^4)_k^{2r-1}}
\pmod{\Phi_d(q^2)}.
\end{align*}
This proves \eqref{eq:sum-to-0}. By Lemma~\ref{thm:NPq2}, we conclude that \eqref{eq:guo1-new} is true modulo $A_n(q^2)C_n(q^2)$.
\end{proof}

We collected adequate ingredients and are able to prove Theorem~\ref{conj:guo}.

\begin{proof}[Proof of Theorem {\rm\ref{conj:guo}}]
For $0\leqslant k \leqslant n-2$, the $q$-factorial $(-q^{2k+2};q^2)_{n-k-1}$ contains the factor $1+q^{2n-2}$, and
for $k=n-1$, we have
\begin{align*}
{2k\brack k}_{q^2}={2n-2\brack n-1}_{q^2}=(1+q^{2n-2}){2n-3\brack n-2}_{q^2}.
\end{align*}
Therefore, the left-hand sides of \eqref{eq:guo1} and \eqref{eq:guo2} are both divisible by $(1+q^{2n-2})^{2r-1}$.

In what follows, we shall prove that the left-hand sides of \eqref{eq:guo1} and \eqref{eq:guo2} are both divisible by $[n]_{q^2}{2n-1\brack n-1}_{q^2}$.
It is easy to see that the $q$-binomial coefficient ${n\brack k}_q$ has the the following factorization (see \cite[Lemma 1]{CH}):
\begin{align*}
{n\brack k}_{q}=\prod_{{d\in\mathcal{D}_{n,k}}}\Phi_d(q),
\end{align*}
where
\begin{align*}
\mathcal{D}_{n,k}:=\left\{d\geqslant 2:
\left\lfloor\frac{k}{d}\right\rfloor+\left\lfloor\frac{n-k}{d}\right\rfloor<\left\lfloor\frac{n}{d}\right\rfloor\right\}.
\end{align*}
It is easily seen that $1<d\in \mathcal{D}_{2n-1,n-1}$
is odd if and only $d\in \mathcal{S}(n)$, and so
\begin{align}
[n]_{q^2}{2n-1\brack n-1}_{q^2}
=A_n(q^2)C_n(q^2)\prod_{{d\,|\,n}\atop{d\geqslant2 \text{~is even}}}\Phi_d(q^2)
\cdot\prod_{{d\in\mathcal{D}_{2n-1,n-1}}\atop{d\text{~is even}}}\Phi_d(q^2). \label{eq:factor}
\end{align}

Note that
\begin{align*}
{2k\brack k}_{q^2}(-q^{2k+2};q^2)_{n-k-1}^2=\frac{(q^2;q^4)_k}{(q^4;q^4)_k}(-q^2;q^2)_{n-1}^2.
\end{align*}
By Theorem \ref{conj:guo-new}, the left-hand sides of \eqref{eq:guo1} and \eqref{eq:guo2} are both divisible by $A_n(q^2)C_n(q^2)$.

It remains to show that \eqref{eq:guo1} and \eqref{eq:guo2} also hold modulo
\begin{align*}
\prod_{{d\,|\,n}\atop{d\geqslant2 \text{~is even}}}\Phi_d(q^2)
\cdot\prod_{{d\in\mathcal{D}_{2n-1,n-1}}\atop{d\text{~is even}}}\Phi_d(q^2).
\end{align*}
Firstly, let $d\,|\,n$ be an even integer. Then
\begin{align*}
1+q^d=\frac{1-q^{2d}}{1-q^d}\equiv0\pmod{\Phi_d(q^2)}.
\end{align*}
It follows that, for $0\leqslant k< n-d/2$, the $q$-shifted factorial $(-q^{2k+2};q^2)_{n-k-1}$
contains the factor $1+q^{2n-d}$ and is therefore congruent to $0$ modulo $\Phi_d(q^2)$.
On the other hand, for $n-d/2\leqslant k\leqslant n-1$,
we have $d\in\mathcal{D}_{2k,k}$, i.e.,
\begin{align}
{2k\brack k}_{q^2}\equiv0 \pmod{\Phi_d(q^2)}.\label{eq:2kkq2}
\end{align}
Hence, for $0\leqslant k\leqslant n-1$, we always have
\begin{align}
{2k\brack k}_{q^2}(-q^{2k+2};q^2)_{n-k-1}\equiv0 \pmod{\Phi_d(q^2)}.  \label{eq:2kkq3}
\end{align}
This proves that \eqref{eq:guo1} and \eqref{eq:guo2} are true modulo $\prod_{{d\,|\,n}\atop{d\geqslant2 \text{~is even}}}\Phi_d(q^2)$.
Secondly, we consider the case where $d\in\mathcal{D}_{2n-1,n-1}$ is even.
Write $n=ud+v$ with $0\leqslant v\leqslant d-1$. Then $v>d/2$.
Thus, for $0\leqslant k< ud+d/2$, the polynomial $(-q^{2k+2};q^2)_{n-k-1}$
has the factor $1+q^{2ud+d}$ and is divisible by $\Phi_d(q^2)$.
Moreover, for $ud+d/2\leqslant k\leqslant n-1$, we have $d\in\mathcal{D}_{2k,k}$,
and so \eqref{eq:2kkq2} holds.
Therefore, for $0\leqslant k\leqslant n-1$, the $q$-congruence \eqref{eq:2kkq3} also holds in this case.
This proves that \eqref{eq:guo1} and \eqref{eq:guo2} are true modulo $\prod_{{d\in\mathcal{D}_{2n-1,n-1}}\atop{d\text{~is even}}}\Phi_d(q^2)$.

Noticing that
\begin{align*}
[n]_{q^2}{2n-1\brack n-1}_{q^2}
=(1+q^{2n-2})[2n-1]_{q^2}{2n-3\brack n-2}_{q^2},
\end{align*}
and $(1+q^{2n-2})$ is relatively prime to $[2n-1]_{q^2}{2n-3\brack n-2}_{q^2}$, the least common multiple
of $(1+q^{2n-2})^{2r-1}$ and $[n]_{q^2}{2n-1\brack n-1}_{q^2}$ is just $(1+q^{2n-2})^{2r-2}[n]_{q^2}{2n-1\brack n-1}_{q^2}$.
This completes the proof of the theorem.
\end{proof}

\section{Proof of Theorem \ref{conj:guo2}}

We first give the following result.
\begin{lem}\label{lem:sd+t}Let $d$ be a positive odd integer. Let $s$ and $t$ be non-negative integers with $0\leqslant t\leqslant d-1$. Then
\begin{align*}
(-q;q)_{sd+t}\equiv2^s(-q;q)_t \pmod{\Phi_d(q)}.
\end{align*}
\end{lem}
\pf It is easy to see that
$(-q;q)_{d-1}\equiv 1\pmod{\Phi_d(q)}$ (see, for example, \cite[Lemma 3.2]{GW}).
Hence, for $0\leqslant t\leqslant d-1$, we have
\begin{align*}
(-q;q)_{sd+t}&=(-q^{sd+1};q)_t\prod_{j=0}^{s-1}(-q^{jd+1};q)_d\\[5pt]
&\equiv(-q;q)_t(-q;q)_d^s\\[5pt]
&\equiv2^s(-q;q)_t \pmod{\Phi_d(q)},
\end{align*}
where we have used $q^d\equiv 1\pmod{\Phi_d(q)}$.
\qed

We also need the following result, which is similar to Lemma \ref{thm:NPq2} and is a special case of \cite[Theorem 2.1]{NP}.
\begin{lem}\label{thm:NP}
Let $\nu_0(q),\nu_1(q),\ldots$ be a sequence of rational functions in $q$
such that, for any positive odd integer $d>1$,
\begin{itemize}
\item[{\rm(i)}] $\nu_k(q)$ is $\Phi_d(q)$-integral for each $k\geqslant0$, i.e.,
the denominator of $\nu_k(q)$ is relatively prime to $\Phi_d(q)$;

\item[{\rm(ii)}] for any non-negative integers $s$ and $t$ with $ 0\leqslant t\leqslant d-1$,
\begin{align*}
\nu_{sd+t}(q)\equiv\mu_s(q)\nu_{t}(q)
\pmod{\Phi_d(q)},
\end{align*}
where $\mu_s(q)$ is a $\Phi_d(q)$-integral rational function only dependent on $s$;

\item[{\rm(iii)}]
\begin{align*}
\sum_{k=0}^{d-1}\frac{(q;q^2)_k}{(q^2;q^2)_k}\nu_k(q)\equiv 0 \pmod{\Phi_d(q)}.
\end{align*}
\end{itemize}
Then, for all positive integers $n$,
\begin{align}
\sum_{k=0}^{n-1}\frac{(q;q^2)_k}{(q^2;q^2)_k}\nu_k(q)\equiv 0
\pmod{A_n(q)C_n(q)}.  \label{eq:main-lem}
\end{align}
\end{lem}

We can now prove  Theorem \ref{conj:guo2}.
\begin{proof}[Proof of Theorem {\rm\ref{conj:guo2}}]
Similarly as before, for $0\leqslant k \leqslant n-2$, we have
\begin{align*}
(-q^{k+1};q)_{n-k-1}
\equiv 0 \pmod{1+q^{n-1}},
\end{align*}
and for $k=n-1$, there holds
$
{2k\brack k}_{q}\equiv 0
\pmod{1+q^{n-1}}.
$
This means that the left-hand side of \eqref{eq:conj-guo2} is divisible by $(1+q^{n-1})^3$.

In what follows, we shall prove that the left-hand side of \eqref{eq:conj-guo2} is divisible by $[n]{2n-1\brack n-1}$.
Letting $q^2\mapsto q$ in \eqref{eq:factor}, we have
\begin{align*}
[n]{2n-1\brack n-1}
=A_n(q)C_n(q)\prod_{{d\,|\,n}\atop{d\geqslant2 \text{~is even}}}\Phi_d(q)
\cdot\prod_{{d\in\mathcal{D}_{2n-1,n-1}}\atop{d\text{~is even}}}\Phi_d(q).
\end{align*}
For any non-negative integer $k$, let
\begin{align*}
\nu_k(q)=q^{-{k+1\choose 2}}[3k+1]\frac{(q;q^2)_{k}^2(-q;q)_k^2}{(q^2;q^2)_{k}^2}.
\end{align*}
Then, applying Lemmas \ref{lem:NP} and \ref{lem:sd+t}, for any positive odd integer $d$ and non-negative integers $s$ and $t$ with $0\leqslant t\leqslant d-1$, we have
\begin{align*}
\nu_{sd+t}(q)&=q^{-{sd+t+1\choose 2}}[3(sd+t)+1]\frac{(q;q^2)_{sd+t}^2(-q;q)_{sd+t}^2}{(q^2;q^2)_{sd+t}^2}\\
&\equiv \frac{1}{4^s}{2s\choose s}^2\nu_t(q)
\pmod{\Phi_d(q)}.
\end{align*}
Moreover, by \eqref{eq:0phi-q}, we obtain
\begin{align*}
\sum_{k=0}^{d-1}q^{-{k+1\choose 2}}[3k+1]\frac{(q;q^2)_{k}^3(-q;q)_k^2}{(q^2;q^2)_{k}^3}
\equiv 0\pmod{\Phi_d(q)}.
\end{align*}
Thus, we may apply Lemma \ref{thm:NP} to deduce that
\begin{align}
\sum_{k=0}^{n-1}q^{-{k+1\choose 2}}[3k+1]\frac{(q;q^2)_{k}^3(-q;q)_k^2}{(q^2;q^2)_{k}^3}
\equiv 0 \pmod{A_n(q)C_n(q)}. \label{eq:mod-ac}
\end{align}
Multiplying the left-hand side of \eqref{eq:mod-ac} by $(-q;q)_{n-1}^4$,
we conclude that \eqref{eq:conj-guo2} is true modulo $A_n(q)C_n(q)$.

It remains to show that \eqref{eq:conj-guo2} is also true modulo
$$
\prod_{{d\,|\,n}\atop{d\geqslant2 \text{~is even}}}\Phi_d(q)
\cdot\prod_{{d\in\mathcal{D}_{2n-1,n-1}}\atop{d\text{~is even}}}\Phi_d(q).
$$
This is exactly same as the proof of Theorem \ref{conj:guo} and is omitted here.
\end{proof}

\vskip 2mm \noindent{\bf Acknowledgments.}
The first author was partially supported by the National Natural Science Foundation of China (grant 11771175).
The second author was partially supported by the Natural Science Foundation of Inner Mongolia, China (grant 2020BS01012).


\begin{thebibliography}{99}
\small \setlength{\itemsep}{-.8mm}

\bibitem{Bauer}G. Bauer, Von den coefficienten der Reihen von Kugelfunctionen einer variabeln, J. Reine Angew. Math. 56 (1859), 101--121.

\bibitem{BB}J.M. Borwein and P.B. Borwein, Pi and the AGM, volume 4 of Canadian Mathematical Society Series
of Monographs and Advanced Texts, John Wiley \& Sons, Inc., New York, 1998.


\bibitem{CH} William Y. C. Chen and Q.-H. Hou, Factors of the Gaussian coefficients, Discrete Math., 306 (2006), 1446--1449.

\bibitem{Guillera}J. Guillera, A method for proving Ramanujan's series for $1/\pi$, Ramanujan J. 52 (2020), 421--431.

\bibitem{Guo2020A.2}V.J.W. Guo, A $q$-analogue of the (A.2) supercongruence of Van Hamme for primes $p\equiv1 \pmod{4}$,
Rev. R. Acad. Cienc. Exactas F\'is. Nat., Ser. A Mat. RACSAM 114 (2020), Art. 123.

\bibitem{Guo-rima}V.J.W. Guo, Proof of some $q$-supercongruences modulo the fourth power of a
 cyclotomic polynomial, Results Math. 75 (2020), Art. 77.


\bibitem{Guo-mod4}V.J.W. Guo, $q$-Supercongruences modulo the fourth power of a cyclotomic polynomial via creative microscoping,
Adv. Appl. Math. 120 (2020), Art. 102078.

\bibitem{Guo2020}V.J.W. Guo, $q$-Analogues of two ``divergent'' Ramanujan-type supercongruences, Ramanujan J. 52 (2020), 605-624.

\bibitem{Guo2018}V.J.W. Guo, A $q$-analogue of a Ramanujan-type supercongruence involving central
binomial coefficients. J. Math. Anal. Appl. 458 (2018), 590--600.

\bibitem{GS3}V.J.W. Guo and M.J. Schlosser, A new family of $q$-supercongruences modulo the fourth power of a cyclotomic polynomial,
Results Math. 75 (2020), Art. 155.

\bibitem{GS}V.J.W. Guo and M.J. Schlosser,
Some $q$-supercongruences from transformation formulas for basic hypergeometric series,
Constr. Approx., to appear.

\bibitem{GW} V.J.W. Guo and S.-D. Wang, Some congruences involving fourth powers of central $q$-binomial coefficients, Proc. Roy. Soc. Edinburgh Sect. A 150 (2020), 1127--1138.


\bibitem{GuoZu}V.J.W. Guo and W. Zudilin, A $q$-microscope for supercongruences, Adv. Math. 346 (2019), 329--358.

\bibitem{GuoZu0}V.J.W. Guo and W. Zudilin, Ramanujan-type formulae for $1/\pi$: $q$-analogues, Integral Transforms Spec. Funct. 29 (2018), 505--513.

\bibitem{HMZ}Q.-H. Hou, Y.-P. Mu, and D. Zeilberger, Polynomial reduction and supercongruences, J. Symbolic Comput.  103 (2021), 127--140.

\bibitem{LW}L. Li and S.-D. Wang, Proof of a $q$-supercongruence conjectured by Guo and Schlosser,
Rev. R. Acad. Cienc. Exactas Fs. Nat., Ser. A Mat. RACSAM 114 (2020), Art. 190.

\bibitem{Liu}J-C. Liu, On a congruence involving $q$-Catalan numbers, C. R. Math. Acad. Sci. Paris 358 (2020), 211--215.

\bibitem{LP}J.-C. Liu and F. Petrov, Congruences on sums of $q$-binomial coefficients, Adv. Appl. Math. 116 (2020), 102003.

\bibitem{MZ}G.-S. Mao and T. Zhang, Proof of Sun's conjectures on super congruences and the divisibility of certain binomial sums,
Ramanujan J. 50 (2019), 1--11.

\bibitem{Mortenson}E. Mortenson, A $p$-adic supercongruence conjecture of van Hamme, Proc. Amer. Math. Soc. 136 (2008), 4321--4328.

\bibitem{NP}H.-X. Ni and H. Pan, Divisibility of some binomial sums, Acta Arith. 194 (2020), 367--381.

\bibitem{NP2}H.-X. Ni and H. Pan, Some symmetric $q$-congruences modulo the square of a cyclotomic polynomial,
J. Math. Anal. Appl. 481 (2020), Art. 123372.

\bibitem{Olive}G. Olive, Generalized powers, Amer. Math. Monthly 72 (1965), 619--627.

\bibitem{Ramanujan}S. Ramanujan, Modular equations and approximation to $\pi$, Quart. J. Math. 45 (1914), 350--372.

\bibitem{Sun}Z.-W. Sun, Super congruences and Euler numbers, Sci. China Math. 54 (2011), 2509--2535.

\bibitem{Sun2013}Z.-W. Sun, Products and sums divisible by central binomial coefficients, Electron. J. Combin. 20 (2013), \#P9.

\bibitem{VHamme}L. Van Hamme, Some conjectures concerning partial sums of generalized hypergeometric series,
in: $p$-Adic Functional Analysis (Nijmegen, 1996), Lecture Notes in Pure and Appl. Math. 192, Dekker, New York, 1997, pp.~223--236.

\bibitem{WY}X. Wang and M. Yu, Some new $q$-congruences on double sums,
Rev. R. Acad. Cienc. Exactas F\'is. Nat., Ser. A Mat. RACSAM 115 (2021), Art. 9.

\bibitem{WY0}X. Wang and M. Yue, Some $q$-supercongruences from Watson's $_8\phi_7$ transformation formula,
Results Math. 75 (2020), Art. 71.

\bibitem{Zud2009}W. Zudilin, Ramanujan-type supercongruences, J. Number Theory 129 (2009), 1848--1857.

\bibitem{Zudilin2}W. Zudilin, Congruences for $q$-binomial coefficients, Ann. Combin. 23 (2019), 1123--1135.

\end{thebibliography}
\end{document}